\documentclass{article}
\usepackage[utf8]{inputenc}

\usepackage{amsmath}
\usepackage{amsfonts, amssymb}

\newenvironment{proof}[1][] {\noindent {\bf Proof#1:} }{\hspace*{\fill}$\square$\medskip\par}

\newtheorem{thm}{Theorem}
\newtheorem{lem}[thm]{Lemma}
\newtheorem{prop}[thm]{Proposition}

\def\R{{\mathbb R}}
\def\E{{\mathbb E\,}}
\def\P{{\mathbb P}}
\def\AA{\mathcal A}
\def\compl{\textrm {c}}

\def\NN{\mathcal N}
\def\SS{\mathcal S}

\def\be#1\ee{\begin{equation}#1\end{equation}}
\def\ed#1{ {\mathbf 1}_{ \{#1  \}}}             % indicator
\def\Var{\textrm{Var}\,}

\title{Gaussian Assignment Process}
\author{M.A.Lifshits, A.A.Tadevosian }
\date{}

\begin{document}

\maketitle
\begin{abstract}
We define Gaussian assignment process, determine the asymptotic behavior of its maximum's expectation and suggest an explicit strategy that attains the corresponding asymptotics.
\end{abstract}

\section{Introduction} \label{Introduction}

We consider the following \textit{random assignment problem}.  Let ($X_{ij}$) be an $n\times n$ random matrix with i.i.d. random entries having a common distribution $\mathcal{P}$. Let $\SS_n$ denote the group of permutations $\pi : \{1, 2, \dots, n\} \to \{1, 2, \dots, n\}$. For every $\pi\in \SS_n$ let
\[
   S(\pi)=\sum\limits_{i=1}^{n} X_{i\pi(i)}.
\]

We are interested in the study of $\min\limits_{\pi\in \SS_n} S(\pi)$ or $\max\limits_{\pi\in \SS_n} S(\pi)$ and in finding the optimal permutation 
$ \arg\min\limits_{\pi\in \SS_n} S(\pi)$,
resp. 
$\arg\max\limits_{\pi\in \SS_n} S(\pi)$.

%% \ref{fig:bipartite}). 
%%\begin{figure}[H]
%%    \centering
%%    \includegraphics[scale=1.5]{img2.png}
%%    \label{fig:bipartite}
%%\end{figure}

We refer to \cite{ CoppersmithSorkin,SteeleProbTheoryAndCombOpt} for many applications of assignment problem in various fields of mathematics.
\medskip

The setting with $(X_{ij})$ uniformly distributed on $[0, 1]$ was studied by Steele \cite{SteeleProbTheoryAndCombOpt} and M\'ezard and Parisi \cite{ RandomLinkMatching}, where the authors 
proved that
\[
    \E{\min\limits_{\pi\in \SS_n} S(\pi)} = \zeta(2) - \frac{\zeta(2) + 2\zeta(3)}{n} + O\left(\frac{1}{n^2}\right), \text{ as } n\to \infty,
\]
$\zeta(\cdot)$ being  Riemann's zeta function.
M\'ezard and Parisi \cite{SpinGlassTheory} also conjectured that in the exponential case  
($\mathcal{P} = \mathrm{Exp}(1)$) it is true that
\be \label{mezpar}
    \E{\min\limits_{\pi\in \SS_n} S(\pi)} \to \zeta(2) = \frac{\pi^2}{6}.
\ee
Using replica method from statistical physics \cite{DotsenkoReplica:1993}, they provided an heuristical
argumentation in favor of this conjecture.

Later Parisi \cite{Parisi98aconjecture} conjectured the following explicit expression for every fixed $n$
 \[
    \E{\min\limits_{\pi\in \SS_n} S(\pi)} =\sum\limits_{k=1}^{n} \frac{1}{k^{2}}.
\]
and confirmed it for $n=1,2$ and for $n\to\infty$. Dotsenko \cite{Dotsenko_2000}  
also investigated the precise solution in the exponential case.

In \cite{SorkinMN,BuckChan,CoppersmithSorkin} similar problems were investigated for rectangular matrices.

Aldous \cite{AldousZet2}  gave a rigorous proof of  M\'ezard--Parisi conjecture \eqref{mezpar}.
His approach is based on the assignment analysis of a graph with edges provided with exponentially distributed
weights, see \cite{AldousAsymp}.
\medskip

We are interested here in the case when the distribution $\mathcal{P}$ is the standard normal, i.e. 
$\mathcal{P} = \NN\left(0, 1\right)$. By convenience reasons, for such symmetric distributions it is more natural to study the \textit{maximum}  of random assignment.  
 
 We stress that the support of Gaussian distribution is unbounded which essentially changes the results. Now the expectation of maximum does not tend to a finite limit, as $n\to\infty$,  but increases to infinity, although quite slowly. 

Our main result is the following theorem:

\begin{thm}\label{t:upperlower_bounds}
Let $\{S(\pi),\pi\in\SS_n\}$ be a Gaussian process given by
\be \label{Gass}
    S(\pi) = \sum_{i=1}^{n}  X_{i\pi(i)}, \qquad \pi\in \SS_n,
\ee
where $X_{ij}\,\,(1\leq i,j \leq n)$ are i.i.d. standard Gaussian random variables.
Then it is true that
\be \label{main_opt}
    \lim\limits_{n\to\infty} \frac{\E{\max\limits_{\pi\in \SS_n}S(\pi)}}{n\sqrt{2\log{n}}} = 1.
\ee
\end{thm}
\medskip

In the following we call $\{S(\pi),\pi\in\SS_n\}$ defined in \eqref{Gass} {\it a Gaussian assignment process}.
It seems to be quite an interesting object worth of detailed studies in its own right. Note that $S$ is stationary with respect to the group structure of $\SS_n$. 
\bigskip

We also show that this asymptotic behavior  of the maximum is attained at an explicitly constructed
 {\textit{greedy}} random permutation $\pi^*$ that turns out to be asymptotically optimal 
 in terms of expectation, i.e.
\be \label{greedy_opt}
    \lim\limits_{n\to\infty}\frac{\E{S(\pi^*)}}{n\sqrt{2 \log{n}}} = 1.
\ee
In addition to the expectation study, we provide a central limit theorem for  $S(\pi^*)$,
namely, the following is true.

\begin{thm}\label{t:clt} We have
\[
   \frac{S(\pi^*)-A_n}{B_n} \xrightarrow[n \to \infty]{d} \NN(0, 1),
\]
where
\begin{eqnarray*}
   A_n &=& n\sqrt{2 \log{n}} + O\left(\frac{n \log{\log{n}}}{\sqrt{\log{n}}}\right),
\\
    B_n^2 &=& \frac{\pi^2}{12} \frac{n}{\log{n}} + o\left(\frac{n}{\log{n}}\right).
\end{eqnarray*}
Moreover,
\[
   \sup_{r\in\R}  \left| \P\left\{ \frac{S(\pi^*)-A_n}{B_n}  \le r\right\}-\Phi(r) \right| = O\left(\frac{1}{\sqrt{n}}\right),
\]
where $\Phi(\cdot)$ denotes the standard normal distribution function.
\end{thm}

The structure of the work is as follows.
In Section \ref{s:upper} we provide an upper bound
for maximum's expectation.
In Section \ref{s:greedy} we describe the greedy strategy (Section \ref{ss:greedy_def}) and provide the lower bound for its outcome thus proving its optimality (Section \ref{ss:greedy_lower}). Finally, 
the corresponding central limit theorem  (Theorem \ref{t:clt}) is proved in Section \ref{ss:proof_clt}.

\section{An upper bound}
\label{s:upper}
We use the following standard estimate for the maximum of Gaussian random variables, see \cite[p.180]{GRF}.

\begin{lem}
Let $\{X_j\}_{j=1}^N$ be a family of centered Gaussian random variables such that $\max\limits_{1\le j\le N} \E X_j^2\le \sigma^2$. Then
\be \label{upper}
    \E \max\limits_{1\le j\le N} X_j \le \sqrt{2\log N}\ \sigma.
\ee
\end{lem}
 
 We stress that no assumptions on the dependence are required in this statement.
 
Since $\E S(\pi)^2=n$ for all $\pi\in \SS_n$ and $|\SS_n|=n!$, we obtain from \eqref{upper} the necessary upper bound
\[
  \E \max\limits_{\pi\in \SS_n} S(\pi) \le \sqrt{2\log (n!)\ n}\
  = n \sqrt{2\log n} + O(n),  \qquad \textrm{as } n\to \infty.
\]

\section{The greedy strategy and its properties}
\label{s:greedy}
\subsection{Definition}
\label{ss:greedy_def}
Consider the following \textit{greedy} strategy for constructing a random permutation $\pi^*$
providing an asymptotically optimal (in average) value of the assignment process. Let 
$[i] := \{1, 2, \dots, i\}$.
Define
\[
    \pi^*(1) := \arg \max\limits_{j\in[n]} X_{1j},
\]
and let for all  $i=2,\dots, n$
\[
    \pi^*(i) := \arg \max\limits_{j \not\in \pi^*([i-1])} X_{ij}.
\]

It is natural to call this strategy greedy, because on every step we consider the line $i$, take 
the maximum of its available elements (without considering the influence of this choice on
subsequent steps) and then forget the line $i$ and the corresponding column  $\pi^*(i)$.

Due to the simple structure, the summands in the representation
\be \label{Gass_greedy}
    S(\pi^*) = \sum_{i=1}^{n}  X_{i\pi^*(i)}
\ee
are independent.

\subsection{A lower bound}
\label{ss:greedy_lower}
Our main goal in this subsection is summarized in the following statement.

\begin{prop}\label{p:greedy}
It is true that
\[
 \E{S(\pi^*)} \geq n\sqrt{2 \log{n}} 
 \, (1+ o(1)), 
 \qquad \textrm{as } n\to \infty.
\]
\end{prop}

\begin{proof}

For every fixed $i$ introduce the index set
\[ 
   \beta_{n,i}:=[n]\backslash \pi^*([i-1]).
\]
and denote $m = m(n,i):= |\beta_{n,i}| = n-i+1$. Notice that $m$ random variables
\[
    \{X_{ij},\ j\in \beta_{n,i}\}
\]
are still i.i.d. standard normal although the index set $\beta_{n,i}$ is itself random. 

Denote $\AA:= \{\max\limits_{j \in \beta_{n,i}} X_{ij} \geq 0\}$ 
and let  $\AA^\compl$ be its complement.

The expectation $\E{X_{i\pi^*(i)}}$ can be split into two parts:
\[
    \E{X_{i\pi^*(i)}} = \E{\max\limits_{j \in \beta_{n,i}} X_{ij}} 
    = \E{\max\limits_{j \in \beta_{n,i} } X_{ij} \ed{\AA} } 
      + \E{\max\limits_{j \in \beta_{n,i}} X_{ij} \ed{\AA^\compl}}.
\]

The second term is small because for every $k \in \beta_{n,i}$

\begin{eqnarray*}
    \Big|\E{\max\limits_{j \in\beta_{n,i} } X_{ij} \ed{\AA^\compl}}\Big| 
  &\le& 
  \E \left(|X_{ik}| \ed{X_{ij}\le 0, j \in \beta_{n,i},j\not=k}\right)
\\
    &=&  \left(\frac{1}{2}\right)^{m-1} \cdot\E{\left|X_{11}\right|}  = o(1).
\end{eqnarray*}

Consider the first term
\[
    \E{\max\limits_{j \in \beta_{n,i} } X_{ij} \ed{\AA}} 
    =  \int\limits_{0}^{\infty} \P\left\{\max\limits_{j \in \beta_{n,i} } X_{ij} \ed{\AA} > t\right\} dt.
\]

It is obvious that
\[
    \P\left\{\max\limits_{j \in \beta_{n,i} } X_{ij} \ed{\AA} \leq t\right\}
    \leq 
    \P\left\{\max\limits_{j \in \beta_{n,i} } X_{ij} \leq t\right\}
    =
    \Phi^m(t).
\]

Therefore, for every $r>0$ it is true that
\[
    \E{\max\limits_{j \in \beta_{n,i} } X_{ij} \ed{\AA}} 
    \geq  \int\limits_{0}^{\infty} \left(1 - \Phi^m(t)\right) dt 
    \geq \int\limits_{0}^{r} \left(1 - \Phi^m(t)\right) dt 
    \geq r \left(1 - \Phi^m(r)\right).
\]

We have the following upper bound for $\Phi^m(r)$:
\[
    \Phi^m(r) = (1 - \Hat{\Phi}(r))^m  \leq \exp{\left(-m \Hat{\Phi}(r)\right)},
\]
where $\Hat{\Phi}$ is the tail of the standard normal law. By using inequality
\[
    \Hat{\Phi}(r) \geq \frac{1}{\sqrt{2\pi}} \left(\frac{1}{r} - \frac{1}{r^3}\right) e^{-r^2 / 2}
\]
for $r := r(m) = \sqrt{2 \log{m}} - 1$ we have
\[
    \exp{(-m\Hat{\Phi}(r))} = o(1),
\]
hence,
\[
    \E{\max\limits_{j \in \beta_{n,i} } X_{ij} \ed{\AA}} \geq \sqrt{2 \log{m}}\,\,(1 + o(1)),
\]
and we arrive at
\[
    \E{X_{i\pi^*(i)}} \geq \sqrt{2 \log{m}} \,\, \left(1 + o\left(1\right)\right) + o(1).
\]

Therefore we have a lower bound
\begin{eqnarray*}
%%    \E{\max\limits_{\pi \in \SS_n} S(\pi)} \geq 
\E{S(\pi^*)} &=& \sum\limits_{i=1}^{n} \E{X_{i\pi^*(i)}}
\\
    &\geq& \sum\limits_{i=1}^{n} \sqrt{2 \log{m(n, i)}}\,
           (1 + o(1))
\\
    &=& \sum\limits_{m=1}^{n} \sqrt{2 \log{m}}\, (1+o(1))
\\
    &=& n\, \sqrt{2 \log{n}}\, (1+o(1)).
\end{eqnarray*}

\end{proof}

 Taken together, the upper bound \eqref{upper} and Proposition \ref{p:greedy}  yield the chain of bilateral estimates:
\[
    n\sqrt{2\log{n}}\, (1+o(1)) 
    \leq \E{S(\pi^*)} 
    \leq  \E \max\limits_{\pi\in\SS_n} S(\pi) 
    \leq  n\sqrt{2\log{n}}\, (1+o(1)), 
    \text{ as } n\to\infty,
\]
which proves both \eqref{main_opt} and \eqref{greedy_opt}. 

We conclude that the greedy strategy is asymptotically optimal for maximization of expectation of the Gaussian assignment process.

\subsection{Central Limit Theorem}
\label{ss:proof_clt}

\begin{proof}[ of Theorem \ref{t:clt}]

For proving a central limit theorem for sums of independent variables, it is sufficient
to prove that the corresponding \textit{Lyapunov fraction} 
%% with $\delta = 1$ 
tends to zero.
%%In our case, one should check that
%%\[
%%    \lim\limits_{n\to \infty} \frac{1}{B_n^3} %%\sum\limits_{i=1}^{n} \E{\Big|X_{i\pi^*(i)} - %%\E{X_{i\pi^*(i)}}\Big|^3} = 0.
%%\]
%%which we will do now.
Recall that each term $X_{i\pi^*(i)}$ of the sum
in the representation \eqref{Gass_greedy} is  a maximum of 
$m = m(n, i) := n - i + 1$ 
independent standard normal random variables. It is well known that properly centered and scaled Gaussian maxima converge weakly to Gumbel distribution, namely,
\[
    \frac{X_{i\pi^*(i)} - a_m}{b_m} \xrightarrow[m \to \infty]{d} G,
\]
where $G$ has the distribution function $F_G(x) = \exp(-e^{-x})$,
and
\begin{eqnarray*}
    a_m &:=& \sqrt{2\log{m}} - \frac{\log{\log{m}}+\log{4\pi}}{2\sqrt{2\log{m}}}, \\
    b_m &:=& (2\log{m})^{-\frac12}.
\end{eqnarray*}
see, e.g. \cite[Sect.\,2.3.2]{Galambosh}.

It is known (see. \cite{Gumbel}) that all moments of  $G$ are finite. In particular,  $\E{G} = \gamma, \Var{G} = \zeta(2)$, where $\gamma$ is Euler constant and $\zeta(\cdot)$ is Riemann's zeta function.

Furthermore, an elementary calculation shows that the family of random variables 
$\left\{\Big|\frac{X_{i\pi^*(i)} - a_m}{b_m}\Big|^3\right\}_{n,i}$
is uniformly integrable.
Therefore, the weak convergence implies convergence of moments,
see \cite[Sect.\,1.5]{Billingsley}.
\begin{eqnarray*}
 \frac{\E X_{i\pi^*(i)}-a_m}{b_m} 
    &\xrightarrow[m \to \infty]{}& 
    \E{G};
\\
 \E{\left(\frac{X_{i\pi^*(i)}-a_m}{b_m}\right)^2} 
    &\xrightarrow[m \to \infty]{}& 
    \E{G^2};
\\
    \E{\left|\frac{X_{i\pi^*(i)}-a_m}{b_m}\right|^3} 
    &\xrightarrow[m \to \infty]{}& 
    \E{|G|^3}.
\end{eqnarray*}

We obtain asymptotic expressions for expectations and variances of
$X_{i\pi^*(i)}$, namely:
\begin{eqnarray} \label{formula:expect}
   \E{X_{i\pi^*(i)}} &=& a_m + b_m \E{G} (1+o(1))= a_m + b_m \gamma(1+o(1)),
\\ \nonumber
   \Var {X_{i\pi^*(i)}} &=&  b_m^2 \left[ \E{\left(\frac{X_{i\pi^*(i)}-a_m}{b_m}\right)^2} 
   - \left(\frac{\E X_{i\pi^*(i)}-a_m}{b_m}\right)^2 \right]
\\   \nonumber
     &=& b_m^2 \left[ \E G^2 - (\E G)^2 \right] (1+o(1)) 
\\  \nonumber
     &=&  b_m^2 \Var G \  (1+o(1)) 
      = \left(2 \log{m}\right)^{-1}  \zeta(2) \, (1+o(1)).
%%\\
%%   \E{|X_{i\pi^*(i)}-a_m|^3} &=& b_m^3 \E{|G|^3} (1+o(1)).\nonumber
\end{eqnarray}

By using \eqref{formula:expect} we also obtain the bound
\[
      \left|a_m - \E{X_{i\pi^*(i)}}\right| = O\left(b_m\right).
\]

Finally, let us evaluate the third absolute central moment of $X_{i\pi^*(i)}$:
\begin{eqnarray*}
    \E{\left|X_{i\pi^*(i)} - \E{X_{i\pi^*(i)}}\right|^3}
    &=&
    \E{ \left|X_{i\pi^*(i)} - a_m + a_m - \E{X_{i\pi^*(i)}}\right|^3}
\\
   &\leq& 4 \, \E{\left|X_{i\pi^*(i)} - a_m \right|^3}
      + 4 \left|a_m - \E{X_{i\pi^*(i)}}\right|^3
\\
    &=& O(b_m^3)  = O \left(  \left( \log{m}\right)^{-3/2} \right).
\end{eqnarray*}

It follows that
\begin{eqnarray*}
    \sum\limits_{i=1}^{n} \E{\left|X_{i\pi^*(i)} - \E{X_{i\pi^*(i)}}\right|^3}
    &\leq&  \sum\limits_{i=1}^{n} O\left(   \left( \log{m(n,i)}\right)^{-3/2}  \right)
\\
    &=&  \sum\limits_{m=1}^{n} O\left( \left( \log{m}\right)^{-3/2}  \right)
\\
    &=&   O\left( n  \left( \log{n}\right)^{-3/2}  \right).
\end{eqnarray*}

Consider now the variance of $S(\pi^*)$
\[
  B_n^2 := \Var S(\pi^*) =\sum_{i=1}^n \Var X_{i\pi^*(i)}.
\]
We have
\begin{eqnarray*}
    B_n^2   &=&  
    \sum\limits_{m=1}^{n} \zeta(2) \left(2 \log{m}\right)^{-1} (1+o(1))
\\
    &=& \frac{\zeta(2)}{2}
       \sum\limits_{m=1}^{n} \frac{ 1+o(1)}{\log{m}}
     =  \frac{\pi^2}{12}  \  \frac{ n( 1+o(1))}{\log{n}}.
\end{eqnarray*}

Therefore, for Lyapunov fraction we have a bound 
\[
  L_n:= B_n^{-3/2} \sum\limits_{i=1}^{n} \E{\left|X_{i\pi^*(i)} - \E{X_{i\pi^*(i)}}\right|^3} = O(n^{-1/2}), \quad \textrm{as } n\to \infty.
\]
Now Lyapunov's Central Limit Theorem yields
\[
   \frac{S(\pi^*)-A_n}{B_n} \xrightarrow[n \to \infty]{d} \NN(0,1).
\]
with $A_n:=\E S(\pi^*)$ and $B_n^2=\Var S(\pi^*)$ as defined above.
Moreover, we have Berry--Esseen bound for the convergence rate: there exists a numerical constant $C>0$ such that for all $n$ it is true that

\begin{eqnarray*}
    \sup_{r\in\R}  \left| \P\left\{ \frac{S(\pi^*)-A_n}{B_n}  \leq r\right\} -\Phi(r) \right|
    &\leq&   C L_n = O(n^{-1/2}).
\end{eqnarray*}

Finally, we have the following asymptotic expression for $A_n$.
\begin{eqnarray*}
    A_n &=& \sum_{i=1}^{n} \E{X_{i\pi^*(i)}} = \sum_{m=1}^{n} (a_m + b_m \gamma(1+o(1)))
\\
    &=&     \sum_{m=1}^{n} \left[ \sqrt{2\log{m}} + O\left(\frac{\log{\log{m}}}{\sqrt{\log{m}}}\right)
    + \frac{\gamma(1+o(1))}{ (2 \log{m})^{1/2}}  \right]
\\
    &=&     n\sqrt{2 \log{n}} + O\left(\frac{n \log{\log{n}}}{\sqrt{\log{n}}}\right).
\end{eqnarray*}
 
 \end{proof}
 
 \section*{Acknowledgement} 
 The work supported by Russian Science Foundation Grant 21-11-00047.
 
 %%%%%%%%%%%%%%%%%%%%%%%%%%%%%%%%%%%%%%%%%%%%%

\end{document}